\documentclass[10]{amsart}
\usepackage{amssymb, amsmath}
\usepackage{graphicx}
\input xy
\xyoption{all}
\newtheorem{thm}{Theorem}[section]
\newtheorem{cor}[thm]{Corollary}
\newtheorem{lem}[thm]{Lemma}
\newtheorem{notat}[thm]{Notation}
\newtheorem{fac}[thm]{Fact}

\newtheorem{prop}[thm]{Proposition}
\newtheorem{rema}[thm]{Remark}
\newtheorem{problem}[thm]{Problem}

\theoremstyle{definition}

\theoremstyle{remark}

\numberwithin{equation}{section}

\newcommand{\bea}{\begin{eqnarray*}}
\newcommand{\eea}{\end{eqnarray*}}
\newcommand{\zz}[1]{}

\newtheorem{theorem}{Theorem}[section]

\newtheorem{fact}[theorem]{Fact}

\newcommand{\hbx}{\hfill$\Box$}

\newcommand{\FF}{{\mathbb{F}}}

\newcommand{\ct}{{\rm ct}}

\newcommand{\rank}{{\rm rank}}

\begin{document}

\title{On the number of simplices required to triangulate a Lie group}%
\author{Haibao Duan$^*$}
\address{Institute of Mathematics, Chinese Academy of Sciences, Beijing 100190, China}
\email{dhb@math.ac.cn}%
\thanks{$^{*}$ Supported by  National Natural Science Foundation of China  (No. 11961131004)}

\author{Wac{\l}aw Marzantowicz$^{**}$}
\address{Faculty of Mathematics and Computer Science, Adam Mickiewicz University of Pozna\'{n}, ul.
Uniwersytetu Pozna{\'{n}}skiego 4, 61-614 Pozna\'n, Poland}%
\email{marzan@amu.edu.pl}
\thanks{$^{**}$ Supported by the Polish  Research Grant NCN   Sheng 1 UMO-2018/30/Q/ST1/00228}

\author{Xuezhi Zhao$^*$}
\address{Department of Mathematics, Capital Normal University, Beijing 100048, China}
\email{zhaoxve@mail.cnu.edu.cn}

\zz{\thanks{\phantom{\;\;\;Dedicated
to Daciberg Lima Gon\c calves with best wishes of $2^2 \, 5^2$ years}}}

\thanks{\;\;\;Dedicated to Daciberg Lima Gon{\c{c}}alves with best wishes of
$2^2 \cdot  5^2$ years}


\begin{abstract}
We estimate the number of simplices required for triangulations of
compact Lie groups. As in the previous work  \cite{GMP2}, our
approach combines the estimation of the number of vertices by means
of the covering type via a cohomological argument from \cite{GMP},
and application of the recent versions of the  Lower Bound Theorem
of combinatorial topology. For the exceptional Lie groups, we
present a complete calculation using the description of their
cohomology rings given by the first and third authors. For the
infinite  series of classical Lie groups of growing dimension $d$,
we estimate the growth rate of number of simplices of the highest
dimension, which extends to the case of simplices of (fixed)
codimension $d - i$.
\end{abstract}

 \maketitle
\section{Introduction}

 Every smooth manifold $X$ admits an (essentially
unique) compatible piecewise linear structure, i.e. a triangulation.
Obviously it is finite if $X$ is compact. But the existence theorem
says nothing about the number of simplices, e.g.  vertices, we need
for a triangulation of $X$. The problem of finding minimal
triangulation, i.e. a triangulation which has minimal number of
vertices,  was a subject of many studies of combinatorial topology
(see \cite{Bag-Dat} and \cite{Lutz} for references). Consequently,
it is important to give any estimate of the number of vertices, or
more generally the number of simplices of given  dimension $0\leq i
\leq d=\dim X$.

\begin{notat}\label{defn: f_i numbers}
For  a  triangulation of a $d$-dimensional closed manifold by a
simplicial complex $K$, we  denote by $f_i$, $i=0,\ldots,d$ the
number of $i$-dimensional simplices in $K$.

\end{notat}

In this work we present estimates of coordinates of the vector
$(f_0, \,f_1\,,\dots\,, f_d)$ in the case where $X$ is  a compact
Lie group. We restrict our study  to classical Lie groups for which
the cohomology rings have complete description. The case of five
exceptional Lie groups is equipped with complete calculation of the
formula based on the complete description of their cohomology rings
given by the first and third author in \cite{DZ3}. For the remaining
infinite series of classical Lie groups we describe the asymptotic
growth of the number $f_i$  of vertices of dimension $i$, according
to the dimension $d$. According to our knowledge there is no result
in literature which gives  estimate of the number of simplices of
triangulation of Lie groups in general. Let us recall that every
compact Lie group has finite-sheeted covering which is a product of
a torus and some simple and simply-connected Lie groups, see
\cite[App. 1.2]{Whitehead}. Simple and simply-connected Lie groups
have four infinite series $A_n$, $B_n$, $C_n$, $D_n$, and finite
family of exceptional Lie groups $G_2$, $F_4$, $E_6, E_7,E_8$. The
series A, B, C, and D correspond to the groups $SU(n+1)$,
$SO(2n+1)$, $Sp(n)$, and $SO(2n)$ respectively.

Our approach has two factors. In \cite{K-W} M. Karoubi and Ch.
Weibel  defined a  homotopy invariant of a space $X $ called the
\emph{covering type} of $X$ and denoted  $\ct(X)$. Directly from the
 definition  it follows that $\ct(X)$ is a lower bound for $f_0(X)$ (cf. \cite{K-W}, also
\cite{GMP}). In  \cite{GMP}, a method  estimating $\ct(X)$ from
below
 was presented. The main result  of this paper provides a
formula in terms of multiplicative structure of the cohomology ring
$H^*(X;R)$ in any coefficient ring $R$. More precisely, it estimates
$\ct(X)$ by the maximal  weighted length of a non-zero product in
 $\tilde{H}^*(X;R)$ (Theorem \ref{thm:ct}).

The second component of our approach is based on the recent much more
 sharper versions of the Lower Bound Theorem, shortly called LBT (see \cite{GMP2} for an exposition).
  Purely combinatorial in arguments LBT (cf.\cite{Kalai}) states that the number $i$-dimensional simplices
 of $K$ grows as $f_0$ times the number of $i$-dimensional simplices
 of the standard simplex $\Delta_d$ lowered by a term $ i \,
 {d+1 \choose i} $, which does not depend on $f_0$. Recent
 versions of LBT, called GLBT (cf. \cite{Klee-Novik}, \cite{Novik}, \cite{Novik-Swartz-I}, and \cite{Novik-Swartz-II}, \cite{Novik-Swartz-III})
 or $g$-conjecture confirmed in \cite{Adiprasito}, increase the
 formula of LBT by adding terms which depend on the reduced Betti
 numbers of $X$ (cf.  Theorem \ref{thm:SLBTM}). The latter not only involves the  topology of $X$ but also
 essentially improves the estimate.

The paper is organized in the following way. In the secend  section
we derive or estimate the main formula  of \cite{GMP}  (Theorem
\ref{thm:ct})  for the classical Lie groups estimating the covering
type of spaces in problem. Next, in the third section we  adapt
Theorem \ref{thm:SLBTM} to the discussed spaces  by substituting the
estimate of $f_0$ from  second  section  and  the values (or
estimates) of Betti numbers of studied spaces. At the end we include
a Mathematica notebook which derives the value of main formula
provided values of $f_0$ and the  Betti numbers $\beta_i$ are known.
We present the result of computation of estimates of $f_i$, $0\leq i
\leq  14$ for the group $G_2$, and $F_4$, leaving the reader a
possibility to compute this for the remaining exceptional compact
Lie groups $E_6$, $E_7$, and $E_8$.

\section{Computations of value of covering type by use of  \ref{equ:estimate of ct} }

In this section derive the value of formula of Theorem \ref{thm:ct}
 spaces   which are the object of our investigation. To do it
we restate the some facts presented already  in \cite{GMP}, derived
the value form description of cohomology rings of exceptional Lie
groups given by the first and third author in \cite{DZ3}. We also
adapt the classical results on the cohomology rings of Lie groups
(for example see \cite{F}). Finally we include the results of direct
computations to  some other spaces.

\begin{thm}[{\rm{\cite[Theorem 3.5]{GMP}}}]\label{thm:ct} If there are elements
$x_{k}\in H^{i_k}(X)$ with $i_k>0$,  $k=1,2,\ldots, l$, such that
$x_{1} \cdot x_{2} \cdot\, \cdots\, \cdot x_{ l} \neq 0 $ then

\begin{equation}\label{equ:estimate of ct}
\ct(X)\ge l+1+\sum_{k=1}^l k \,i_k\,.
\end{equation}
Furthermore, if not all $i_k$'s are equal, then $\ct(X)\ge
l+2+\sum_{k=1}^l k \,i_k$.
\end{thm}

We begin with a general theorem which shows that  for given
simple-connected compact Lie group $G$  the values which are
necessary for formula (\ref{equ:estimate of ct}) are  encoded  in
the Cartan algebra the Weyl roots system of $G$.

It is known \cite[CH.1, \S 7]{Fomenko}  and \cite{F} that the
cohomology ring of a compact Lie group $G$ with  coefficients in a
field $\mathcal{R}$  of characteristic $0$ is of the form
$$ H^*(G; \mathcal{R})= \bigwedge_{\mathcal{R}} [y_{2m_1+1}, y_{2m_2+1}, \, \dots, \,
y_{2m_l+1}]$$ where $0\leq m_1 \leq m_2 \leq \, \cdots\, \leq m_l$
and the sequence $(m_1, m_2, \, \dots \,, m_l)$ is called ``the
rational type'' of $G$. Moreover \cite{F}, it is known that if $G$
is a simple compact Lie group then

\begin{equation}\label{equ:sum rational type}
 {\underset{j=1}{\overset{l}{\sum\,}}} m_j =
\frac{1}{2} (d - l)\,,
\end{equation}
where $l= {\rm rk}\, G$ is  the rank of $G$ (i.e.  the dimension of
maximal torus $\mathbb{T} \subset G$) and $d =\dim G$ is the
dimension of $G$.

\begin{thm}\label{thm:general Lie group}
Let $G$ be a compact simple Lie group and $(m_1, m_2, \, \dots \,,
m_l)$ its rational type.  Then
$$ (*)\;\;\;\;\;\; \ct(G)\ge
l+1+\sum_{j=1}^l j \,  (2m_j+1). $$
\end{thm}

The presented below  version of corollary of Theorem
\ref{thm:general Lie group} is due Maciej Radziejewski whom we would
like to express out thanks.  The original version  had weaker form
where the right hand side depended only on $l$

\begin{cor}\label{cor:ct by rank}
Let $G$ be a compact simple Lie group of rank $l$ and dimension $d$.
Let $\theta=  \{ \frac{d-l}{2l} \}$ be the fractional part of
$\frac{d-l}{2l}$.   Then
$$ \ct(G)\,\geq \,\frac{(l+1)(d+2)}{2} + (\theta + \theta^2) \, l^2 - 2 \theta \,.$$
In particular
$$ \ct(G)\,\geq \,\frac{(l+1)(d+2)}{2}\,.$$
\end{cor}

\begin{proof} It follows from (\ref{equ:sum rational type}), the monotonicity of $m_j$ and Chebyshev's sum
inequality that
$$ \ct(G) \geq \,l+1\, +\, \big(\sum_{j=1}^l \, j\big) \; \sum _{j=1}^l
\, (2m_j+1) \,l^{-1}\, = \\
l+1 \ + \, \frac{l+1}{2} \, (l+d -l) =\frac{1}{2} (l+1) (d+2) \,.$$

To  to get the extra term we need to be a little bit precise
consideration. Let $M_k =  \sum_{j=1}^k \, m_j$, $k = 1, \, \dots\,,
 l$.

By Abel's identity we have
$$ l+1+\sum_{j=1}^l j \,  (2m_j+1) \,=\,  \frac{1}{
2} (l + 1)(l + 2)\, +\, 2l\,M_l  -  2 \sum_{k=1}^{l-1}\, M_k\,.$$
From the monotonicity of $m_j$ we have $M_k/k \leq  M_{k+1}/(k+1)$,
$ k = 1, \, \dots \,,l-1$, hence $ M_k \leq  N_k$ for $k = 1, \,
\dots\,,  l$, where $N_l = M_l = \frac{1}{2} (d -l)$ and  $N_{l-1},
\, \dots\,,N_1$ are defined recursively by $$ N_k = \lfloor k\,
N_{k+1}/(k+1) \rfloor\,, \;\;\;\; k=1, \, \dots\,, l-1\,.$$

Take $q \in \mathbb{Z}$ and $r \in \{ 0, \, \dots\,,   l-1 \}$  such
that $M_l = q\,l + r$ . Then $\theta = \frac{r}{l}$  is the
fractional part of  $\frac{d-l}{2l}  = \frac{M_l}{l}$, and $q
=\frac{d-l}{2l} -\theta$.  We have $N_l = ql + r$. By an induction
argument we get
$$ N_k = \lfloor  k \Big(q(k+1)+r+k+1-l\big)/(k+1)\rfloor =  qk+\lfloor k(r+k+1-l)/(k+1)\rfloor = qk+r+k -l\,, \;\;\; k
\geq  l-r\,,$$ and $N_k = qk$, for  $k \leq  l-r$.

Hence  $  \sum_{k=1}^{l-1} N_k = \frac{(l-1)l}{2}\, q\, +\,
\frac{(r-1)r}{2}$,  and consequently
$$ \ct(G) \geq
\frac{1}{2} (l + 1)(l + 2) + 2l \,\frac{d- l}{2} - (l - 1)\,l\,
\Big(\frac{d -l}{2l} - \frac{r}{l} \Big) + (r -1)r
\\ = \frac{1}{2} (l + 1)(l + 2) + (l + r - 2)r \,,
$$
which is equal to the quantity in the assertion.
\end{proof}

\zz{Since $m_i\leq m_j$ if $i<j$, and $m_j\geq 0$, we have
$$ \ct(G)\ge
l+1+\sum_{j=1}^l j \,  (2m_j+1)  \geq  l+1 + \, \sum_{j=1}^l j  =
\frac{(l+1)(l+2)}{2}.$$
\end{proof}}

In particular Corollary \ref{cor:ct by rank} implies
\begin{equation}\label{estim:lower}
\ct(G) \geq \frac{1}{2} (l + l)(d + 2)
\end{equation}
 In the  in the worst case $(d
= l)$ this gives the estimate $\ct(G) \geq \frac{1}{2} (l + l)(l +
2)$. However if $d$ is of the order $l^2$, then   $\frac{1}{2} (l +
1)(d + 2)$ is much larger than $\frac{1}{2} (l + 1)(l + 2)$. The
extra term $(\theta + \theta^2)\,l^2 - 2 \theta\, l$ is always
nonnegative, and provides a further saving of the average size about
$\frac{5}{6}\, l^2$.

\begin{rema}\label{rem:optimalization}\rm
As a matter of fact for  the discussed simply-connected,  simple
compact Lie groups either $2l\mid (d-l)$ thus $\Theta =0$ or  $
\Theta =\frac{1}{2}$ (see \ref{fact: data of Lie}). The first
happens  if $G=SO(2n)$ and $SO(2n+1)$, the  second  happens  if $G=
U(n)$, $SU(n)$, or $Sp(N)$ depending on a parity of $n$.

For the exceptional Lie groups $G_2, \, F_4,\, E_6, \, E_7, \, E_8$
the pairs $(d,l)$ are equal to  $(14, 2),\, (52, 4), \, (78, 6),\,
(133, 7),\, (248, 8) $ respectively. This gives $\Theta = 0$   for
the all exceptional connected compact Lie groups.

Furthermore, using the equality $2M_l= d-l$ and direct calculation
we get
\begin{equation}\label{estim:upper}
 \;\; \frac{1}{2}(l+1)(l+2) + 2\,l\, M_l \;<\;
(l+1)(d+2)\,
\end{equation}
so the upper bound above is of the same order as the lower bound of this expression given in
\ref{estim:lower}.

Consequently, since for each series of  the simple simply connected
connected Lie groups the growth of dimension  $d$ is quadratic in
the rank  $l$ (cf. \ref{fact: data of Lie}), the  estimate of growth
of $\ct(G)$ given by Theorem \ref{thm:general Lie group} is of order
$d^{\frac{3}{2}}$.
\end{rema}

\zz{\begin{rema}\label{rem:optimalization}{\rm A natural question
about of finding  the maximum and minimum of $ \,(*)\,$  reduces to
do the problem of estimation of sum $ \sum_{j=1}^l j m_j = (l+1)M_l
- (M_1 + M_2 + \, \dots\,+ M_l)\,,$ where $\,M_j = m_1+ .
\, \dots\,. +m_j\,$. The above follows from the Abel's identity.   \\
Since $ {\underset{j=1}{\overset{l}{\sum\,}}} m_j = \frac{1}{2} (d -
l)\,,$ the  value of  $M_l$ is fixed.
\begin{itemize}
\item[1)]{ Note that the sum
 $(*)$ is monotonic  with respect to each $m_j$. Consequently the
 minimal value is  $0$, for $m_1= .\,\dots \,. =m_l=0$. If (\ref{equ:sum rational type}) is not satisfied, then there is no the greatest value, in general.}

\item[2)] { If (\ref{equ:sum rational type}) holds and  $ M_l$ is fixed,
then the greatest value is for  $m_1=\cdots =m_{l-1}=0, \,
m_l=M_l$.}

\item[3)] { Concerning once more the lowest value in more detail. From the monotonicity of  $ m_j$ we have $ M_j
\leq \lfloor  j M_{j+1} / (j+1)  \rfloor, j = 1,\, \dots\,,l-1\,.$
Consequently, the equality, thus the lowest value, is for $ m_l =
m_{l-1} =.\, \dots\, = m_k = m_{k-1}+1 =\, \dots \,= m_1+1\,.$
Moreover $m_1 = \lfloor M_l / l \rfloor $ and $ k = l + l m_1 - M_l
+ 1\,$.}
\end{itemize}}
\end{rema}}

Note that $m_1=\, \cdots\, m_{l-1}= 0$ and $m_l=M_l$ implies
$l+1+\sum_{j=1}^l j \,  (2m_j+1) = l+1 + l (2m_l+1)$. This and $(*)$
give $$ l+1+\sum_{j=1}^l j \,  (2m_j+1) = l+1+ l [(d-l) +1]= 1+l +
l(d-l)+ l = l \cdot d  - l^2 + 2\,l +1.$$

\begin{problem}\label{prob:first}
Is there a formula expressing, for a compact Lie group $G$, the
dimension $d=\dim \,G$ in terms of $l={\rm rk}\, G$ of rank of $G$?
It would be enough to know the rate of growth of $d$ as a function
of $l$. In all known examples $d$ is a quadratic  function of $l$.

\zz{

 To study
 {\phantom{ space}} \\
\begin{itemize}
\item[i)]{Give more explicit description of the numbers $m_j$? I am sure that it is known}
\item[ii)]{Give an estimate of the sum of Theorem \ref{thm:general Lie
group} which shows that the value of  $\ct(G)$ is estimated  from
below by a quadratic (or higher order) polynomial in $ \dim
\mathbb{T}$, where  $\mathbb{T} \subset G$ the maximal torus?}
\end{itemize}}

\end{problem}

A classical result coming from the fact $\pi_2(G)=0$ and
$\pi_3(G)=\mathbb{Z}$ for any compact Lie group states the following
(see \cite[Theorem 2.6]{F})
\begin{thm}
Let $G$ be a compact simply-connected simple Lie group. Then
$H^3(G;\mathbb{Q}) \simeq \mathbb{Q}$. This implies that $m_1 = 1$
and $m_i > 1$ for $i > 1$.

Consequently, $H^*(G;\, \mathbb{Q}) = \bigwedge[x_3,\, x_{2m_2+1},
\, \dots\,, \, x_{2m_l+1}]$
\end{thm}

\zz{ The  fact that  Theorem  3.5 of \cite{GMP} holds for any
cohomology theory let us use also the following result
(\cite[Theorem 2.11]{F})
\begin{thm}
 Let $(2m_1 +
1, \, \dots\,, 2m_l + 1)$  be the type of a compact, simply-
connected, simple Lie group $G$. If $ m_ l < \min(p\,m_2,  p^2 -1)$
for a prime $p$, then $H^*(G;\mathbb{Z})$ has no $p$-torsion and
$H^*(G; F_p) = \wedge[x_3,\, x_{2m_2+1}, \, \dots\,, \,
x_{2m_l+1}]$.
\end{thm}}

As a corollary we get the following estimate
\begin{cor}\label{cor:ct of simply-connected}
Let $G$ be a compact simply-connected simple Lie group,
$\mathbb{T}^l\subset G$ its maximal torus, and $(m_1, m_2, \, \dots
\,, m_l)$ its rational type. Then
$$ \ct (G) \geq 3 + (2m_2+ 1) \, + \, \dots\, +\, (2m_l +1)\, + l+1. $$

\end{cor}

\zz{\begin{fac}\label{rem: rate of growth}

 Note that estimates of Theorem \ref{thm:general Lie group} and
Corollary \ref{cor:ct of simply-connected} give only at most cubic
rate of growth of $\ct(G)$ in the dimension of $G$.

Moreover, the same argument shows that in general the estimate
(\ref{equ:estimate of ct})  of Theorem \ref{thm:ct} gives only cubic
rate of growth  of $\ct (X)$ in $n=\dim X$.
\end{fac}

\begin{proof}

Indeed, doing very crude estimation, i.e. replacing $l= \dim
\mathbb{T}$ by $d = \dim G $ and each $2m_j+1 $ by $n$, or
respectively $l$ by $n =\dim X$ and also each  $i_k$ by $n=\dim X$
in the general case, we get the following

$$ \;\;\;\;\;\;
l+1+\sum_{j=1}^l j \,  (2m_j+1) \, \leq d+1 \, +\,
\frac{d^2(d+1)}{2} = (d+1)( 1 +\frac{1}{2} d^2)
$$
or, respectively
 $$ \;\;\;\;\; \;\;\;  l+2+\sum_{k=1}^l k
\,i_k \,\leq \, n+2+\sum_{k=1}^n k \,n = n+2 \,+\,
\frac{n^2(n+1)}{2}\;\;\;\;$$
\end{proof}}

\medskip
\subsection{Classical Lie groups}

The simplest   compact Lie group is the torus $\mathbb{T}^n$, which
is not simply-connected but has nice form of  its cohomology
algebra.
\begin{prop}[\cite{GMP}]\label{prop:torus}

$$ \ct(\mathbb{T}^n) \; \geq \frac{(n+1)(n+2)}{2}.$$
\end{prop}
\begin{proof} Indeed $H^*(\mathbb{T}^n; \mathbb{R}) = \bigwedge[t_1, \, \dots\, t_n]$, where $\deg t_i= 1$. Thus $ \ct(\mathbb{T}^n) \geq 1 \cdot 1 \, + 1 \cdot 2 \,
+\, \cdots \,+ 1 \cdot n\, + n+1 = \frac{(n+1)(n+2)}{2}\,.$ Of
course the statement also follows from Corollary \ref{cor:ct by
rank}
\end{proof}

Analogous argument works for the unitary groups as well.  But in
this case the knowledge  of degrees of generators of $H^*( \;\;;
\mathbb{Z})$ allows  to get better estimate than that in Corollary
\ref{cor:ct by rank}
\begin{prop}\label{prop:U(n)}
The covering type of unitary group is estimated as
$$\ct(U(n)) \geq  \frac{1}{6}(4n^3+3n^2+5n+12) \ \ \text{and} \ \ \ \ct (SU(n)) \geq  \frac{1}{6}(4n^3-3n^2+5n+6).$$
\end{prop}
\begin{proof}
The cohomology algebra $H^*(U(n); \mathbb{Z})$ is the exterior
algebra on generators in dimensions \newline  $1,3,\ldots,(2n-1)$,
while $H^*(SU(n);\mathbb{Z})$ is the exterior algebra on generators
in dimensions $3,5, \ldots,(2n-1)$. Theorem \ref{thm:ct} and the
classical sums
$$\sum_{j=1}^l j \, = \frac{l(l+1)}{2} \;\;\; {\text{and}}\;\;\;
\sum_{j=1}^l \,j^2 \, = \, \frac{l(l+1)(2l+1)}{6}$$ give
$$\ct(U(n)) \geq  1+ 2\cdot 3+3\cdot 5+\ldots+ n\cdot (2n-1) +(n+2)=
 \frac{1}{6}(4n^3+3n^2+5n+12)$$
and
$$\ct(SU(n)) \geq  1\cdot 3+ 2\cdot 5+\ldots+ (n-1)\cdot (2n+1) +(n+1)
= \frac{1}{6}(4n^3-3n^2+5n+6).$$
\end{proof}

\begin{prop}\label{prop:Sp(n)}
The covering type of the symplectic group is estimated  as
$$\ct (Sp(n)) \geq
\frac{1}{6} \,(8·n^3 + 9·n^2 + 7·n + 12) 
  $$

\end{prop}
\begin{proof} We have
$H^*(Sp(n);\mathbb{Z}) = \bigwedge[x_3, \, x_7,\, \dots\, ,\,
x_{4n-1}]$ (cf. \cite{F}) and the statement follows by the same
arguments as Proposition \ref{prop:U(n)} applied to the sum
$\sum_{j=1}^n \, j\, (4j-1)$.
\end{proof}

\begin{prop}\label{prop:S0(n)} The covering type of special
orthogonal group $SO(n)$  is estimated as
$$\ct(SO(n))\, \geq \, \frac{1}{6}(4n^3+3n^2+5n+12) $$
\end{prop}

\begin{proof} Let $\mathbb{F}_2$ be a field of
characteristic $2$, e.g.  $\mathbb{Z}_2$ with the ring structure.
 Then from \cite[Theorem 1.18]{F} we have that  $H^*(SO(n); \mathbb{F}_2)$ is the
quotient of polynomial ring $
 \mathbb{F}_2[x_1, \, x_3, \, \dots \, \,
x_{2m-1}]/(x_i^{\alpha_i})$, where $m = \lfloor \frac{n}{2}\rfloor $
and $\alpha_i$ is the smallest power of two such that $i \, \alpha_i
\geq  n $. Since the generator of ideal $(x_i^{\alpha_i})$ is a
multiple of variables $x_i$ in powers which are powers of $2$, the
multiple $ x_1\, x_3\, \cdots \, x_{2m-1}$ is non-zero in $
 \mathbb{F}_2[x_1, \, x_3, \, \dots \, \,
x_{2m-1}]/(x_i^{\alpha_i})$.  Then the statement follows from
Theorem \ref{thm:ct} and the same calculation as for $G= U(n)$ in
Proposition \ref{prop:U(n)}.
\end{proof}
\subsection{Exceptional Lie groups}

For the five simply-connected exceptional Lie groups $G$, the
structure of the algebra $H^{\ast }(G;\mathbb{F}_{p})$ as a module
over the Steenrod algebra $A_{p}$, has been determined by Duan and
Zhao in \cite[Theorem 1.1]{DZ3}. It follows that 
\begin{equation*}
\begin{array}{l}
H^{\ast }(G_{2};\mathbb{F}_{2})=\mathbb{F}_{2}[\zeta _{3}]/\langle
\zeta
_{3}^{4}\rangle \otimes \Lambda _{\mathbb{F}_{2}}(\zeta _{5}), \\
H^{\ast }(F_{4};\mathbb{F}_{2})=\mathbb{F}_{2}[\zeta _{3}]/\langle
\zeta _{3}^{4}\rangle \otimes \Lambda _{\mathbb{F}_{2}}(\zeta
_{5},\zeta
_{15},\zeta _{23}), \\
H^{\ast }(E_{6};\mathbb{F}_{2})=\mathbb{F}_{2}[\zeta _{3}]/\langle
\zeta _{3}^{4}\rangle \otimes \Lambda _{\mathbb{F}_{2}}(\zeta
_{5},\zeta
_{9},\zeta _{15},\zeta _{17},\zeta _{23}), \\
H^{\ast }(E_{7};\mathbb{F}_{2})={\mathbb{F}_{2}[\zeta _{3},\zeta
_{5},\zeta _{9}]}/{\langle \zeta _{3}^{4},\zeta _{5}^{4},\zeta
_{9}^{4}\rangle }\otimes \Lambda _{\mathbb{F}_{2}}(\zeta _{15},\zeta
_{17},\zeta _{23},\zeta _{27}),
\\
H^{\ast }(E_{8};\mathbb{F}_{2})={\mathbb{F}_{2}[\zeta _{3},\zeta
_{5},\zeta _{9},\zeta _{15}]}/{\langle \zeta _{3}^{16},\zeta
_{5}^{8},\zeta _{9}^{4},\zeta _{15}^{4}\rangle }\otimes \Lambda
_{\mathbb{F}_{2}}(\zeta
_{17},\zeta _{23},\zeta _{27},\zeta _{29}).%
\end{array}%
\end{equation*}

\begin{equation*}
\begin{array}{l}
H^{\ast}(G_{2};\mathbb{F}_{3})=\Lambda_{\mathbb{F}_{3}}(\zeta_{3},%
\zeta_{11}), \\
H^{\ast}(F_{4};\mathbb{F}_{3})=\mathbb{F}_{3}[x_{8}]/\left\langle
x_{8} ^{3}\right\rangle
\otimes\Lambda_{\mathbb{F}_{3}}(\zeta_{3},\zeta_{7}
,\zeta_{11},\zeta_{15}), \\
H^{\ast}(E_{6};\mathbb{F}_{3})=\mathbb{F}_{3}[x_{8}]/\left\langle
x_{8}^{3}\right\rangle
\otimes\Lambda_{\mathbb{F}_{3}}(\zeta_{3},\zeta_{7}
,\zeta_{9},\zeta_{11},\zeta_{15},\zeta_{17}), \\
H^{\ast}(E_{7};\mathbb{F}_{3})=\mathbb{F}_{3}[x_{8}]/\left\langle
x_{8}^{3}\right\rangle \otimes\Lambda_{\mathbb{F}_{3}}(\zeta_{3},\zeta_{7},%
\zeta_{11},\zeta_{15},\zeta_{19},\zeta_{27},\zeta_{35}), \\
H^{\ast}(E_{8};\mathbb{F}_{3})=\mathbb{F}_{3}[x_{8},x_{20}]/\left\langle
x_{8}^{3},x_{20}^{3}\right\rangle
\otimes\Lambda_{\mathbb{F}_{3}}(\zeta
_{3},\zeta_{7},\zeta_{15},\zeta_{19},\zeta_{27},\zeta_{35},\zeta_{39},%
\zeta_{47}).%
\end{array}
\end{equation*}
\begin{equation*}
\begin{array}{l}
H^{\ast}(G_{2};\mathbb{F}_{5})=\Lambda_{\mathbb{F}_{5}}(\zeta_{3},\zeta
_{11}), \\
H^{\ast}(F_{4};\mathbb{F}_{5})=\Lambda_{\mathbb{F}_{5}}(\zeta_{3},\zeta
_{11},\zeta_{15},\zeta_{23}), \\
H^{\ast}(E_{6};\mathbb{F}_{5})=\Lambda_{\mathbb{F}_{5}}(\zeta_{3},\zeta
_{9},\zeta_{11},\zeta_{15},\zeta_{17},\zeta_{23}), \\
H^{\ast}(E_{7};\mathbb{F}_{5})=\Lambda_{\mathbb{F}_{5}}(\zeta_{3},\zeta
_{11},\zeta_{15},\zeta_{19},\zeta_{23},\zeta_{27},\zeta_{35}), \\
H^{\ast}(E_{8};\mathbb{F}_{5})=\mathbb{F}_{5}[x_{12}]/\left\langle
x_{12}^{5}\right\rangle \otimes\Lambda_{\mathbb{F}_{5}}(\zeta_{3},%
\zeta_{11},\zeta_{15},\zeta
_{23},\zeta_{27},\zeta_{35},\zeta_{39},\zeta_{47}).%
\end{array}
\end{equation*}
\begin{equation*}
\begin{array}{l}
H^{\ast}(G_{2};\mathbb{Q})=\Lambda_{\mathbb{Q}}(\zeta_{3},\zeta _{11}), \\
H^{\ast}(F_{4};\mathbb{Q})=\Lambda_{\mathbb{Q}}(\zeta_{3},\zeta
_{11},\zeta_{15},\zeta_{23}), \\
H^{\ast}(E_{6};\mathbb{Q})=\Lambda_{\mathbb{Q}}(\zeta_{3},\zeta
_{9},\zeta_{11},\zeta_{15},\zeta_{17},\zeta_{23}), \\
H^{\ast}(E_{7};\mathbb{Q})=\Lambda_{\mathbb{Q}}(\zeta_{3},\zeta
_{11},\zeta_{15},\zeta_{19},\zeta_{23},\zeta_{27},\zeta_{35}), \\
H^{\ast}(E_{8};\mathbb{Q})=\Lambda_{\mathbb{Q}}(\zeta_{3},\zeta_{15},\zeta
_{23},\zeta_{27},\zeta_{35},\zeta_{39},\zeta_{47},\zeta_{59}).%
\end{array}
\end{equation*}

\begin{prop}\label{prop:ct of exceptional}
The covering type of the exceptional Lie groups
$G_{2},F_{4},E_{6},E_{7}$ and $E_{8}$ have the lower bounds
$44,259,486,1288$ and $5870$, respectively.

Consequently $44,259,486,1288$ and $5870$ are lower estimates of
$f_0(G_2)$, $f_0(F_4)$, $f_0(E_6)$, $f_0(E_7)$, and $f_0(E_8)$
respectively.
\end{prop}

\begin{proof}
From the cohomology algebras with coefficients in $\mathbb{F}_{2}$
stated above, we obtain the following estimations:

\begin{itemize}
\item $\mathrm{ct}(G_2) \ge (4+2) +1\times 3+ 2\times 3 +3\times 3 + 4\times
5 = 44$.

\item $\mathrm{ct}(F_4) \ge (6+2)+1\times 3+2\times 3+3\times 3+4\times
5+5\times 15+6\times 23 = 259$.

\item $\mathrm{ct}(E_6) \ge (8+2)+1\times 3+2\times 3+3\times 3+4\times
5+5\times 9+6\times 15+7\times 17+8\times 23 = 486$.

\item $\mathrm{ct}(E_7) \ge (13+2)+1\times 3+2\times 3+3\times 3+4\times
5+5\times 5+6\times 5+7\times 9+8\times 9+9\times 9+10\times
15+11\times 17+12\times 23+13\times 27 = 1288$.

\item $\mathrm{ct}(E_8) \ge (32+2)+1\times 3+2\times 3+3\times 3+4\times
3+5\times 3+6\times 3+7\times 3+8\times 3+9\times 3+10\times
3+11\times 3+12\times 3+13\times 3+14\times 3+15\times 3+16\times
5+17\times 5+18\times 5+19\times 5+20\times 5+21\times 5+22\times
5+23\times 9+24\times 9+25\times 9+26\times 15+27\times 15+28\times
15+29\times 17+30\times 23+31\times 27+32\times 29 = 5870$.
\end{itemize}
\end{proof}

Using the cohomology with other coefficients one obtains different
lower bounds, among which $\mathbb{F}_{2}$ coefficient gives the
best estimation.

\medskip

\subsection{K\"{a}hler  or symplectic manifolds}

\begin{thm}\label{thm: ct for Kahler}
If $X$ is a K\"{a}hler manifold, or a closed symplectic manifold of
(real) dimension $2m$, then $\ct(X)\ge (m+1)^2$.
\end{thm}

\begin{proof}
By definition, there is a cohomology class $\omega\in H^2(X)$ such
that $\omega^m\ne 0\in H^{2m}(X)$ (see \cite[4.23 Theorem]{book}).
Consider $m$ copies of $\omega$. From Theorem~\ref{thm:ct}, we
obtain that $\ct(X)\ge m+1+\sum_{k=1}^m 2k = (m+1)^2$.
\end{proof}

\section{Estimates of number of simplices  of given dimension}

In this section we  estimate the number of simplices of a given
dimension, e.g. of facets, and of all simplices that are needed to
triangulate a Lie group or flag manifold. In our  approach we follow
our previous  work \cite{GMP2}. A classical tool for such an
estimation is  the Lower Bound Theorem of  Kalai \cite{Kalai} and
also Gromov \cite{Gromov} (see \cite{GMP2} for more information).
Note that LBT  is purely combinatorial and does not take into
account the homology of the manifold. As in \cite{GMP2} we are able
to obtain better estimates by using a generalized version of LBT.

First observe that the number of all simplices always increases
exponentially with the dimension. In fact, even the minimal
triangulation of the simplest closed manifold, the $d$-dimensional
sphere, has $d+2$ vertices and $2^{d+2}-2$ simplices. However, we
will show that the number of simplices that are needed to
triangulate  Lie groups, respectively flag manifolds, of comparable
dimension is several orders of magnitude bigger.

As we said  LBT  does not take into account the homology of the
manifold, so we are led to consider stronger results. Much of the
research in enumerative combinatorics of simplicial complexes has
been guided by various versions of the so-called $g$-Conjecture.
These are  far-reaching generalizations of the LBT. Of particular
interest to us is the so called Manifold $g$-Conjecture (see
\cite[Section 4]{Klee-Novik}), as it comes with a version of LBT for
manifolds that takes into account the Betti numbers. Recently,
Adiprasito has published a preprint \cite{Adiprasito}, whose results
combined with the work of Novik and Swartz
\cite{Novik,Novik-Swartz-II,Novik-Swartz-III,Swartz} imply that the
Manifold $g$-Conjecture is true.

 The Strong Manifold $g$-Theorem comes with an even
stronger version of the GLBT (cf. \cite[Theorem 4.5]{GMP2}) which we
formulate here for convenience of  the reader.

\begin{thm}\label{thm:SLBTM} Suppose $\FF$ is a field\,\footnote{In the
literature, $\FF$ is typically assumed to be infinite for technical
reasons, however for the purposes of face enumeration this can be
circumvented by passing to $\FF(t)$, preserving the Betti numbers.
This trick was already used for example in \cite[Theorem~4.3]{Schenzel}.} of arbitrary characteristic. Let $M$ be a connected
$d$-dimensional $\FF$-orientable triangulated $\FF$-homology
manifold without boundary. Suppose $\beta_i$, $i=0,\ldots,d$, are
the reduced Betti numbers of $M$ with respect to $\FF$. Then the
following bounds  hold:
$$
\begin{array}{rcl}
  f_i & \ge & f_0\cdot{d+1 \choose i}-i\cdot{d+2\choose i+1}
   +{d+1\choose i+1}\sum_{j=0}^i{i\choose j}\beta_j\\
  & & +\sum_{j=2}^{\lfloor\frac{d+2}2\rfloor}\left[\binom{d+2-j}{d+1-i}-\binom{j}{d+1-i}\right]
  \binom{d+1}{j-1}\beta_{j-1}\
\ \textrm{for\ }i=0,\ldots,d-1
\end{array}
$$
and
$$f_d\ge f_0\cdot d- (d+2)(d-1)+\sum_{j=0}^{d-1}{d\choose
j}\beta_j+\sum_{j=2}^{\lfloor\frac{d+2}2\rfloor}(d+2-2j)\binom{d+1}{j-1}\beta_{j-1}.$$
\end{thm}

\subsection{Asymptotic estimate for the classical Lie groups}

We are in position to formulate a theorem which states that for the
classical compact Lie groups  for any triangulation the number of
facets (the simplices of highest dimension of the group) grows
exponentially in the rank of groups, thus sub-exponentially in the
dimension of these groups.

To make a use of Theorem \ref{thm:SLBTM} we should know the Betti
numbers of a given space.  The most common and useful for
computation is a presentation of them as the coefficients of
Poincar\'{e} polynomial of $X$. Let $\beta_i$ be the dimension of
$H^i(X; \mathbb{F})$ of the cohomology group in the coefficient in a
field $\mathbb{F}$. The Poincar\'{e} polynomial is defined as:
$$ P(X)(t):= 1 + \beta_1(X) \, t + \, \beta_2(X) \, t^2 \, + \, \cdots\, +
\, \beta_d(X) \, t^d +\, \cdots \,, $$ since $\beta_0(X)=1$ if $X$
is connected. We shall use the cohomology in $\mathbb{Q}$
coefficient if no coefficient is specified.

Study of the Betti numbers  of the  Lie groups began even earlier
than a description of their cohomology rings had been known. For a
brief of  its history and many positions of  classical literature we
refer to \cite{Coleman} and \cite{Samelson}. We have the following
description of Poincar\'{e} polynomial (cf.  \cite{Samelson},
\cite{Coleman}).

Let $ G $ be a compact Lie group and $(m_1, m_2, \, \dots \,, m_l)$
its rational type, where $l$ is the rank of $G$. Then its rational
Poincar\'{e} polynomial is given as
\begin{equation}\label{Poincare for classical}
 P(G)(t)\; = \; \prod_{i=1}^l \, (1 + t^{2m_i +1})\,.
\end{equation}
More precisely
\begin{equation}\label{Poincare specific}
\begin{matrix}
P(SU(n))\, = \, (1+t^3)(1+t^5)\, \cdots \, (1+t^{2n+1})\,,
\;{\text{consequently}}\;\;\cr P(U(n))\,=\, (1+t)\,\cdot\,
(1+t^3)(1+t^5)\cdots (1+t^{2n+1})\,,{\phantom{consequ}} \cr {\text{
and}} \cr  P(SO(2n+1))\,=\,(1+t^3)(1+t^7)\, \cdots \,
(1+t^{4n-1})\,, {\phantom{(1+t^{n})}} \cr \, P(SO(2n)) \,=\,
(1+t^3)(1+t^7)\, \cdots \, (1+t^{4n-5})\cdot (1+t^{2n-1})\,,\cr
 P(Sp(n))\,=\,(1+t^3)(1+t^7)\, \cdots \, (1+t^{4n-1})\,.
{\phantom{lalalalala}}\;\;\;\;\;\;\;\;
\end{matrix}
\end{equation}

We need also an information about  dimensions   and ranks   of the
listed above groups.
\begin{fact}\label{fact: data of Lie}
 For the discussed groups we have
\begin{center}
\begin{tabular}{|c|l|l|l|l|l|}
\hline $G$ & $U(n)$ & $SU(n) $ & $ SO(2n+1)$ & $ SO(2n) $ & $ Sp(n)$
\\ \hline $ d= \dim \,G$ & $ n^2$   & $ n^2-1$  & $n (2n+1)$
& $n (2n-1)$ & $n (n+1)$ \\ \hline $ l= \rank \,G $ & $ n $ & $ n-1
$ & $ n $ & $ n $  & $ n$   \\  \hline
\end{tabular}
\end{center}
\end{fact}

\begin{thm}
\label{thm: exponential facets for classical Lie}
For the classical compact Lie groups $U(n)$, $SU(n)$, $SO(n)$, and
$Sp(n)$ we have the following estimates of number of facets, i.e.
simplices of highest dimension $d=\dim G$:
\begin{itemize}
\item[1)]{$ f_d(U(n)) \,\geq\, \frac{1}{6}\,(4n^5 -3n^4 +5n^3 + 6 n^2
+12)
 \,+ \, 2^n -1 $,}
 \item[2)]{$f_d(SU(n)) \, \geq \,\frac{1}{6}(4n^5-9n^4+n^3+15 n^2+6)\,
 +\, 2^{n-1}-1 $,}
\item[3)]{$f_d(SO(2n+1))\,\geq \frac{1}{3}\, ( 32n^5 + 64n^4 + 64 k^3 + 38k^2 +9k+6)\, + 2^n-1 \,$,}
\item[4)]{$f_d(SO(2n))\, \geq \, \frac{1}{3}\,(
32n^5-16n^4+16 n^3 -2 k^2 - 3k +6) \, +\, 2^n -1$,}
\item[5)]{ $ f_d(Sp(n))\,\geq \, \frac{1}{6}\,
(8n^5+15n^4+12 n^3+11 n^2+ 6 n +12) \, + 2^n -1\,$.}
\end{itemize}
Consequently, in all these cases the number $f_d$ of facets  grows
exponentially in $n$, thus in $l=\rank\, G$  and $d=\dim G$.
\end{thm}

\noindent{\it Proof.} To show the statement  it  is enough to apply
the second formula of Theorem \ref{thm:SLBTM}  and argue as follows.
\begin{itemize}
\item{First substitute to the formula of Theorem \ref{thm:SLBTM}  as $f_0$ the
values of estimates of $\ct$ stated in  Propositions
\ref{prop:U(n)}, \ref{prop:S0(n)}, and \ref{prop:Sp(n)}
respectively.}

\item{Secondly, input to this formula the value of dimension $d$
from the table \ref{fact: data of Lie}.}

\item{Thirdly, truncate the sum of formula of Theorem
\ref{thm:SLBTM} disregarding the last term (sum) of it.}

\item{Next, in the term $\sum_{j=0}^{d-1}{d-1\choose
j}\beta_j$ take into account  only these reduced non-zero Betti
numbers which correspond to the multiples $t^{2m_{i_1} +1}
t^{2m_{i_2}+1}\, \cdots  \, t^{2m_{i_s}+1} $ for multi-indices
$(i_1, \,\dots, i_s)$, $i_k\geq 1$,  of lengths $1 \leq s \leq l$,
where $l= \rank \,G$. Observe that always the upper index of
summation, which is  equal to $d-1$, is greater or equal to the
parameter $n$, or $n-1$ if $G=SU(n)$. The latter is equal to $l$,
the rank of these groups. Consequently $\binom{d-1}{j} \geq
\binom{l}{j}$, and
$$ \sum_{j=0}^{d-1}{i\choose
j}\beta_j \,\geq \sum_{j=0}^{l}{l\choose j}\beta_j\,, $$ where each
$j$ corresponds to a multi-index $(i_1, \,\dots, i_s)$ of length
$1\leq s\leq l$. By the above, $\beta_j \geq 1$. Now the statement
follows from the binomial formula $\sum_{j=0}^{l}{l\choose j}= 2^l$,
since we have to subtract $l\choose 0 = 1$ as the zero Betti
number.} {\hbx}
\end{itemize}

\begin{rema}\label{rem:facets of fixed codimension}\rm
Note that a modification of  argument of Theorem \ref{thm:
exponential facets for classical Lie} let us to show the exponential
rate of growth   of the number $f_{d-i_0}(G)$ of simplices of fixed
codimension $d-i_0$ for all groups of the statement of Theorem
\ref{thm: exponential facets for classical Lie}. As in Theorem
\ref{thm: exponential facets for classical Lie} the varying argument
is $n$ , which can be identified with the rank of $G$.
\end{rema}

\zz{To show it is is enough to apply the formula of Theorem
\ref{thm:SLBTM} for $i=d$ and do the following.
\begin{itemize}
\item{First substitute to the formula of Theorem \ref{thm:SLBTM}  as $f_0$ the
values of estimates of $\ct$ stared in Propositions \ref{prop:U(n)},
\ref{prop:S0(n)}, and \ref{prop:Sp(n)}.}

\item{Secondly, input to this formula the value of dimension $d$
from the table \ref{fact: data of Lie}.}

\item{Thirdly, truncate the sum of formula of Theorem
\ref{thm:SLBTM} disregarding the last term (sum) of it.}

\item{Next, in the term $\sum_{j=0}^i{i\choose
j}\beta_j$ take into account  only these reduced non-zero Betti
numbers which correspond to the multiples $t^{2m_{i_1} +1}
t^{2m_{i_2}}\, \cdots  \, t^{2m_{i_s}} $ for multi-indices $(i_1,
\,\dots, i_s$, $i_k\geq 1$,  of lengths $1 \leq s \leq l$, where $l=
\rank G$.}
\end{itemize}}
At the end of this subsection we turn back to the K{\"{a}}hler and
symplectic manifolds. As a correspondent of Theorem \ref{thm:
exponential facets for classical Lie} we have the following theorem

\begin{thm}\label{thm:facets for Kahler}
If $X$ is a K\"{a}hler manifold, or a closed symplectic manifold of
(real) dimension $2m$, then
$$ f_{2m}(X) \,\geq (2m^3+2m +1) \, + (2^{2(m-1)} -1)=2m^3+2m + 2^{2(m-1)}  $$
\end{thm}
\begin{proof}
Since $f_0\geq (m+1)^2$ by Proposition \ref{thm:facets for Kahler},
and $d=2m$ the first term of the second formula of Theorem~\ref{thm:SLBTM} is equal to $2n^3+2n+1$. Moreover,  $\omega^j \neq
 0$ in $H^{2j}(X;\mathbb{Z})$ for every $ 1\leq j \leq m$, which
 implies that $\beta_{2j}(X) \geq 1$ for $1\leq j\leq m $.  Since
$(1-1)^{2m-1} = 0$, we have  $\sum_{0}^{2m-1} \binom{2m-1}{2j} =
\sum_{1}^{2m-1} \binom{2m-1}{2j-1} = \frac{1}{2}\, 2^{2m-1} =
2^{2(m-1)}$. Now the statement follows by the same argument as in
Theorem \ref{thm: exponential facets for classical Lie}.
\end{proof}

\begin{rema}\label{rem:computation}\rm
We  emphasize that the estimates of Theorems \ref{thm: exponential
facets for classical Lie} and \ref{thm:facets for Kahler} are very
imprecise, especially in low dimensions. For example, the K{\"a}hler
manifolds have  larger Betti numbers for interesting examples of
algebraic varieties (cf. \cite{Sa}). Consequently, for lower
dimensional cases it is better to derive completely the formula of
Theorem \ref{thm:SLBTM} than to estimate only.
\end{rema}

\subsection{Computations for the exceptional Lie groups}
In the last section we present complete calculation of the numbers
$f_i$ for two exceptional groups $G_2$ and $F_4$, and the
Poincar\'{e} polynomials for all  exceptional groups. We also
provide the codes sources for a notebook of Mathematica program
package. This and the data of Poincar\'{e} polynomials allow to
derive the numbers  $f_i$ , $ 0\leq i \leq \dim \,G$ for the
remaining exceptional Lie groups.  However, the length of outputs
are too long to include them in this paper. \newpage
\begin{prop}
The Poincar\'{e} polynomials $P {(G;\mathbb{F})}$ of the five simply-connected exceptional Lie groups $G$ with coefficients in fields
$\mathbb{F}_2,  \mathbb{F}_3,  \mathbb{F}_5,  \mathbb{Q}$ are:
\begin{equation*}
\begin{array}{l}
P{(G_{2};\mathbb{F})}=\left\{
\begin{array}{ll}
\dfrac{(1-t^{12})}{(1-t^{3})}\cdot (1+t^{5}), & \mbox{ if }\mathbb{F}=\mathbb{%
F}_{2}, \\
(1+t^{3})(1+t^{11}), & \mbox{ if }\mathbb{F}=\mathbb{F}_{3},\mathbb{F}_{5},%
\mathbb{Q},%
\end{array}%
\right.  \\
P{(F_{4};\mathbb{F})}=\left\{
\begin{array}{ll}
\dfrac{(1-t^{12})}{(1-t^{3})}\cdot (1+t^{5})(1+t^{15})(1+t^{23}), &
\mbox{
if }\mathbb{F}=\mathbb{F}_{2}, \\
(1+t^{8}+t^{16})(1+t^{3})(1+t^{7})(1+t^{11})(1+t^{15}), & \mbox{ if }\mathbb{%
F}=\mathbb{F}_{3}, \\
(1+t^{3})(1+t^{11})(1+t^{15})(1+t^{23}), & \mbox{ if }\mathbb{F}=\mathbb{F}%
_{5},\mathbb{Q}, \\
\end{array}%
\right.  \\
P{(E_{6};\mathbb{F})}=\left\{
\begin{array}{ll}
\dfrac{(1-t^{12})}{(1-t^{3})}\cdot
(1+t^{5})(1+t^{9})(1+t^{15})(1+t^{17})(1+t^{23}), & \mbox{ if }\mathbb{F}=%
\mathbb{F}_{2}, \\
(1+t^{8}+t^{16})(1+t^{3})(1+t^{7})(1+t^{9})(1+t^{11})(1+t^{15})(1+t^{17}), & %
\mbox{ if }\mathbb{F}=\mathbb{F}_{3}, \\
(1+t^{3})(1+t^{9})(1+t^{11})(1+t^{15})(1+t^{17})(1+t^{23}), & \mbox{ if }%
\mathbb{F}=\mathbb{F}_{5},\mathbb{Q},%
\end{array}%
\right.  \\
P{(E_{7};\mathbb{F})}=\left\{
\begin{array}{ll}
\dfrac{(1-t^{12})(1-t^{20})(1-t^{36})}{(1-t^{3})(1-t^{5})(1-t^{9})}\cdot
(1+t^{15})(1+t^{17})(1+t^{23})(1+t^{27}), & \mbox{ if }\mathbb{F}=\mathbb{F}%
_{2}, \\
(1+t^{8}+t^{16})(1+t^{3})(1+t^{7})(1+t^{11})(1+t^{15})(1+t^{19})(1+t^{27})(1+t^{35}),
& \mbox{ if }\mathbb{F}=\mathbb{F}_{3}, \\
(1+t^{3})(1+t^{11})(1+t^{15})(1+t^{19})(1+t^{23})(1+t^{27})(1+t^{35}), & %
\mbox{ if }\mathbb{F}=\mathbb{F}_{5},\mathbb{Q}, \\
\end{array}%
\right.  \\
P{(E_{8};\mathbb{F})}=\left\{
\begin{array}{ll}
\dfrac{(1-t^{48})(1-t^{40})(1-t^{36})(1-t^{60})}{%
(1-t^{3})(1-t^{5})(1-t^{9})(1-t^{15})}%
(1+t^{17})(1+t^{23})(1+t^{27})(1+t^{29}), & \mbox{ if }\mathbb{F}=\mathbb{F}%
_{2}, \\
(1+t^{8}+t^{16})(1+t^{20}+t^{40})(1+t^{3})(1+t^{7})(1+t^{15})(1+t^{19})
&
\\
\ \ \ \cdot (1+t^{27})(1+t^{35})(1+t^{39})(1+t^{47}), & \mbox{ if }\mathbb{F}%
=\mathbb{F}_{3}, \\
\dfrac{(1-t^{60})}{(1-t^{12})}%
(1+t^{20}+t^{40})(1+t^{3})(1+t^{11})(1+t^{15})(1+t^{23}) &  \\
\ \ \ \cdot (1+t^{27})(1+t^{35})(1+t^{39})(1+t^{47}), & \mbox{ if }\mathbb{F}%
=\mathbb{F}_{5}, \\
(1+t^{3})(1+t^{15})(1+t^{23})(1+t^{27})(1+t^{35})(1+t^{39})(1+t^{47})(1+t^{59}),
& \mbox{ if }\mathbb{F}=\mathbb{Q}. \\
\end{array}%
\right.
\end{array}%
\end{equation*}
\end{prop}

Applying Theorem \ref{thm:SLBTM} we can obtain estimations of the
number $f_{i}$ of simplices of dimension $i$ of the relevant group
$G$. This task can be implemented using Mathematica with the code
enclosed below.

The input is the Poincar\'{e} polynomial $P{(G;\mathbb{F})}$ of a
group $G$, while the output is a table that lists all the
estimations of $f_{i}$, where $i$ ranges from $0$ to $d=\dim G$.

\begin{verbatim}
d = Length[CoefficientList[p, t]] - 1
li = Table[f0*Binomial[d + 1, i]- i*Binomial[d + 2, i + 1]
    + Binomial[d + 1, i + 1]*Sum[Binomial[i, j]*b[j], {j, 0, i}]
    + Sum[(Binomial[d + 2 - j, d + 1 - i]
         - Binomial[j, d + 1 - i])*Binomial[d + 1, j - 1]*b[j - 1],
          {j, 2, (d + 2)/2}],
  {i, 0, d - 1}];
li = Append[li, f0*d - (d + 2)*(d - 1) + Sum[Binomial[d, j]*b[j], {j, 0, d - 1}]
   + Sum[(d + 2 - 2*j)*Binomial[d + 1, j - 1]*b[j - 1], {j, 2, (d + 2)/2}]];
li = li /. Table[b[j] -> Abs[Coefficient[p, t^j]], {j, d}]/. {b[0] -> 0}
\end{verbatim}

\begin{rema}\label{rem: simplices of G_2}\rm  From the Poincar\'{e} polynomial $P{(G_{2};\mathbb{F}_{2})}$
of the group $G_{2}$ with coefficients in $\mathbb{F}_{2}$ we obtain
the following estimations of the $f_{i}(G_{2})$'s, where $1\leq
i\leq \dim G_{2}=14$:
\begin{center}
$44\;,540\;,3500\;,16380\;,60060\;,180180\;,460460\;,1003860\;,1793220\;,$
\\ $ 2494492\;,2582580\;,1901900\;,936740\;,276060\;,36808.$
\end{center}
In comparison, the estimations coming from the Poincar\'{e}
polynomials with other coefficients are much smaller.
\end{rema}
\begin{rema}\label{rem:simplices for F_4}\rm  In the next table we
present estimations of the $f_{i}$'s for the group $F_4$  derived
separately by use of  the Poincar\'{e} polynomial with coefficients
in $\mathbb{F}_{2}$, $\mathbb{F}_{3}$, or $\mathbb{F}_{5}$
respectively. For comparison, in each dimension $i$ the biggest
lower bounds are underlined. Note that the estimations vary
according to the  choices of coefficients.

\begin{tabular}{|c|r|r|r|}
\hline $i$ & $\mathbb{F}_2$ & $\mathbb{F}_3$ & $\mathbb{F}_5$ \\
\hline
0 & $\underline{259}$ & $\underline{259}$ & $\underline{259}$ \\
1 & $\underline{12296}$ & $\underline{12296}$ & $\underline{12296}$ \\
2 & $\underline{307294}$ & $\underline{307294}$ & $\underline{307294}$ \\
3 & $\underline{5434832}$ & $\underline{5434832}$ & $\underline{5434832}$ \\
\ 4 & $\underline{75841675}$ & $\underline{75841675}$ &
$\underline{75841675}
$ \  \\
5 & $\underline{898211405}$ & $872384240$ & $872384240$ \\
6 & $\underline{9665099080}$ & $8425395160$ & $8425395160$ \\
7 & $\underline{98189141960}$ & $70096565630$ & $69056099840$ \\
8 & $\underline{936843104470}$ & $532679948710$ & $484818522370$ \\
9 & $\underline{8207348294600}$ & $4019473489850$ & $2942591397200$ \\
10 & $\underline{65068495431940}$ & $31469774141260$ & $15579780596380$ \\
11 & $\underline{465307900554770}$ & $246812512408280$ & $72908019801890$ \\
\ 12 & $\underline{3006871679028070}$ & $1825555954104190$ &
$312895574628490
$ \  \\
13 & $\underline{17600137873624250}$ & $12284217111466790$ & \ $%
1357046655920900$ \\
14 & $\underline{93440294896349800}$ & $74314230778534600$ & \ $%
6810475068598600$ \\
15 & $\underline{450286771378309775}$ & $403688085014794580$ & \ $%
40032696051173735$ \\
16 & $1971743019982744645$ & $\underline{1972544627081800135}$ & \ $%
241874900220997225$ \\
17 & $7863182337222190160$ & $\underline{8691999543917841515}$ & \ $%
1351282534317191840$ \\
18 & $28667776883962562710$ & $\underline{34659355642150540210}$ & \ $%
6719774511823246510$ \\
19 & $96053339758806527825$ & $\underline{125651509261850608400}$ & \ $%
29616143389087516625$ \\
20 & $297546706965222844675$ & $\underline{416590423956641354875}$ & \ $%
116366415690518027125$ \\
21 & $857048823513321474035$ & $\underline{1271376687212781173570}$ & \ $%
410552931017085466610$ \\
22 & $2305515484583856497680$ & $\underline{3594398180086388474800}$ & \ $%
1308208247349691040080$ \\
23 & $5805916666340273259140$ & $\underline{9464430118366786351880}$ & \ $%
3780300926971447336640$ \\
24 & $13691742470273968084348$ & $\underline{23296033020334540996348}$ & \ $%
9933596706928762306348$ \\
25 & $30203414158511162699228$ & $\underline{53694335012120681194316}$ & \ $%
23781024562728215423648$ \\
26 & $62222956498405788385384$ & $\underline{115866106264425527967208}$ & \ $%
51939670248200359867144$ \\
27 & $119545026801812836973476$ &
$\underline{233700844962180391946332}$ & \
$103609314739176120792736$ \\
28 & $214032931319683517783924$ &
$\underline{439499576816343586980164}$ & \
$188951821406493657563324$ \\
29 & $357111493947308465621620$ &
$\underline{768471351552073128466780}$ & \
$315309645494651593502080$ \\
30 & $555564068183111798392400$ &
$\underline{1245866340397735244442800}$ &
\ $481853242200739273482800$ \\
\hline
\end{tabular}

\begin{tabular}{|c|r|r|r|}
\hline $i$ & $\mathbb{F}_2$ & $\mathbb{F}_3$ & $\mathbb{F}_5$ \\ \hline
31 & $806407339828295811066370$ &
$\underline{1868141433978914243019325}$ &
\ $674893164389212842638110$ \\
32 & $1092420574393401473029355$ &
$\underline{2585238914432460431713295}$ &
\ $867084819996731389618565$ \\
33 & $1380439040299677655288810$ &
$\underline{3295440234015897348233125}$ &
\ $1022779365505852631873560$ \\
34 & $1624553842322905438967750$ &
$\underline{3862524944394191012523950}$ &
\ $1108664551104442752692090$ \\
35 & $1775459827365420580030465$ &
$\underline{4155078440267128287279730}$ &
\ $1105338541605402502408465$ \\
36 & $1794828559813170720089615$ &
$\underline{4093985916498066977355695}$ &
\ $1014223149115635950179325$ \\
37 & $1670073199551033201858115$ &
$\underline{3685614429263730593888440}$ &
\ $856503920963737735154830$ \\
38 & $1422327363482087712114920$ &
$\underline{3022487137167360033316040}$ &
\ $665073826503017218206920$ \\
39 & $1101859031090233520299090$ &
$\underline{2249556580874893485166930}$ &
\ $473731770600839059400350$ \\
40 & $771285427381721886017930$ &
$\underline{1512601618241344821600410}$ &
\ $308274771427740105664130$ \\
41 & $484321425064871781237760$ &
$\underline{913754781405029162953900}$ & \
$182165928067660360915000$ \\
42 & $270669335071872386276900$ &
$\underline{492583625059817580232580}$ & \
$96970862050975415599460$ \\
43 & $133428150140921671860370$ &
$\underline{235027856422964278167130}$ & \
$46039249877048119561870$ \\
44 & $57416514733144586383910$ & $\underline{98266558885147592120090}$ & \ $%
19262564316093857041010$ \\
45 & $21298963639050452896450$ & $\underline{35559965593024259110360}$ & \ $%
7001144224020018463300$ \\
46 & $6705076821048773378120$ & $\underline{10964133899498923237640}$ & \ $%
2172346995036321314120$ \\
47 & $1755155335743793388860$ & $\underline{2821888206480529615975}$ & \ $%
562962158585866488085$ \\
48 & $371543159139563573435$ & $\underline{589525752602412668315}$ & \ $%
118360526603610225935$ \\
49 & $61086632457851987623$ & $\underline{95994644156525073778}$ & \ $%
19372482408185237968$ \\
50 & $7316756676108737390$ & $\underline{11425877633680732118}$ & \ $%
2313937523940263390$ \\
51 & $567831053316383782$ & $\underline{884005590508296811}$ & \ $%
179334499784124337$ \\
52 & $21427586917599388$ & $\underline{33358701528614974}$ & \ $%
6767339614495258$ \\ \hline
\end{tabular}

Our estimations of number of  all simplices, i.e.  of all
dimensions, which  to triangulations of the exceptional Lie groups
are given as follows.

\begin{tabular}{|c|l|l|l|l|}
\hline
$G$ & $\mathbb{F}_2$ & $\mathbb{F}_3$ & $\mathbb{F}_5$ & $\mathbb{Q}$ \\
\hline
$G_2$ & $\underline{11746824}$ & $4059144$ & $4059144$ & $4059144$ \\
$F_4$ & $1.57775\times 10^{25}$ & $\underline{3.50157\times 10^{25}}$ & $%
9.63191\times 10^{24}$ & $9.63191\times 10^{24}$ \\
$E_6$ & $1.66706\times 10^{38}$ & $\underline{2.57662\times 10^{38}}$ & $%
8.22191\times 10^{37}$ & $8.22191\times 10^{37}$ \\
$E_7$ & $1.44756\times 10^{65}$ & $\underline{1.23839\times 10^{73}}$ & $%
1.68159\times 10^{64}$ & $1.68159\times 10^{64}$ \\
$E_8$ & $\underline{1.85929\times 10^{121}}$ & $1.30821\times 10^{120}$ & $%
7.2883\times 10^{119}$ & $1.40822\times 10^{119}$ \\ \hline
\end{tabular}
\end{rema}

\end{document}